\documentclass[11pt]{amsart}
\usepackage[top=1in, bottom=1.25in, left=1.25in, right=1.25in]{geometry} 

\usepackage{url, 
  amssymb,setspace, mathrsfs,fontenc}
\usepackage[alphabetic]{amsrefs}
\usepackage{graphicx}
\usepackage[mathcal]{euscript}
\usepackage{verbatim}

\newtheorem{thm}{Theorem}
\newtheorem{lem}[thm]{Lemma}

\newtheorem{prop}[thm]{Proposition}

\newtheorem{conj}[thm]{Conjecture}

\newtheorem{cor}[thm]{Corollary}
\newtheorem{defe}[thm]{Definition}
 
\theoremstyle{remark}
\newtheorem{rem}[thm]{Remark}
\newtheorem{exam}[thm]{Example}

\DefineSimpleKey{bib}{myurl}
\newcommand\myurl[1]{\url{#1}}
\BibSpec{webpage}{
  +{}{\PrintAuthors} {author}
  +{,}{ \textit} {title}
  +{}{ \parenthesize} {date}
  +{,}{ \myurl} {myurl}
}

\let\oldmarginpar\marginpar
\renewcommand\marginpar[1]{\-\oldmarginpar[\raggedleft\footnotesize #1]%
{\raggedright\footnotesize #1}}

\newcommand{\nc}{\newcommand}

\nc{\ssec}{\subsection}

\nc{\on}{\operatorname}

\nc {\cG} {\mathcal{G}}
\nc {\cK}{\mathcal{K}}
\nc {\cC} {\mathcal{C}}
\nc {\cL} {\mathcal{L}}
\nc {\cE} {\mathcal{E}}
\nc {\cM} {\mathcal{M}}
\nc {\cO}{\mathcal{O}}
\nc {\cF}{\mathcal{F}}
\nc {\cZ}{\mathcal{Z}}

\nc {\bZ}{\mathbb{Z}}
\nc {\bQ}{\mathbb{Q}}

\nc {\uG} {\underline{G}}
\nc {\cB}{\mathcal{B}}
\nc{\rat}{\mathrm{rat}}
      
\nc {\fk}{\mathfrak{k}}
\nc {\fI}{\mathfrak{i}}
\nc {\fg} {\mathfrak{g}}
\nc {\fu} {\mathfrak{u}}
\nc {\fl} {\mathfrak{l}}
\nc {\fn} {\mathfrak{n}}
\nc {\cP} {\mathcal{P}}
\nc {\fz} {\mathfrak{z}}
\nc {\fc}{\mathfrak{c}}
\nc {\fh}{\mathfrak{h}}
\nc {\fp}{\mathfrak{p}}

\nc{\tg} {\mathtt{g}}
\nc {\hfg} {\widehat{\fg}}
\nc {\hG} {\check{G}}

\nc {\bGm} {\mathbb{G}_m}
\nc{\bC}{\mathbb{C}}
\nc{\bV}{\mathbb{V}}
\nc{\bP}{\mathbb{P}}
\nc{\bA}{\mathbb{A}}

\nc{\Sl}{\mathfrak{sl}}

\nc{\ra}{\rightarrow}

\nc {\tU}{\tilde{U}}
\nc {\tSym}{\widetilde{Sym}}

\nc {\Bun}{\mathrm{Bun}}
\nc {\cA}{\mathcal{A}}
\nc {\Fun}{\mathrm{Fun}}
\nc {\crit}{\mathrm{crit}}
\nc {\Ind}{\mathrm{Ind}}
\nc {\Vac}{\mathrm{Vac}}
\nc {\gr}{\mathrm{gr}}
\nc {\ad}{\mathrm{ad}}
\nc {\Sym}{\mathrm{Sym}}
\nc {\Ram}{\mathrm{Ram}}
\nc {\FG}{\mathrm{FG}}
\nc {\Op}{\mathrm{Op}}
\nc {\Hitch}{\mathrm{Hitch}}
\nc {\fb}{\mathfrak{b}}
\nc{\cDt}{\mathcal{D}^\times}
\nc{\cDb}{\mathcal{D}_b^\times}
\nc{\cDbp}{\mathcal{D}_{b'}^\times}
\nc {\gl}{\mathfrak{gl}}
\nc {\Sp}{\mathfrak{sp}}
\nc {\So}{\mathfrak{so}}
\nc {\bR}{\mathbb{R}}
\nc {\Hom}{\mathrm{Hom}}
\nc {\Irr}{\mathrm{Irr}}
\nc {\Id}{\mathrm{Id}}
\nc {\Tr}{\mathrm{Tr}}
\nc {\rk}{\mathrm{rank}}
\nc {\rank}{\mathrm{rank}}
\nc {\cW}{\mathcal{W}}
\nc {\cI}{\mathcal{I}}
\nc {\Fr}{\mathrm{Fr}}
\nc {\ff}{\mathfrak{f}}

\nc {\LocSys}{\mathrm{LocSys}}
\nc {\Ga}{\mathrm{Ga}}
\nc {\ord}{\mathrm{ord}}
\nc{\pole}{\mathrm{pole}}

\newcommand{\C}{\mathbb{C}}
\newcommand{\dtt}{\frac{dt}{t}}
\newcommand{\duu}{\frac{du}{u}}
\newcommand{\n}{\nabla}

\newcommand{\nbr}{[\nabla]}

\newcommand{\Om}{\Omega}
\newcommand{\Ad}{\mathrm{Ad}}

\renewcommand{\a}{\alpha}
\renewcommand{\b}{\beta}
\newcommand{\fs}{\mathfrak{s}}
\newcommand{\Waff}{W^{\mathrm{aff}}}
\newcommand{\bcK}{\bar{\mathcal{K}}}
\newcommand{\Gm}{\mathbb{G}_m}

\DeclareMathOperator{\GL}{GL}
\DeclareMathOperator{\SL}{SL}
\DeclareMathOperator{\Gal}{Gal}

\DeclareMathOperator{\Spec}{Spec}
\DeclareMathOperator{\slope}{slope}
\DeclareMathOperator{\Res}{\mathrm{Res}}

\begin{document} 
\title{A geometric analogue of a conjecture of Gross and Reeder} 
\author{Masoud Kamgarpour}
\author{Daniel S. Sage} 


\address{School of Mathematics and Physics, The University of
  Queensland, Brisbane, QLD 4072, Australia} 
\email{masoud@uq.edu.au}

\address{Department of Mathematics, Louisiana State University, Baton
  Rouge, LA 70808, USA} 
\email{sage@math.lsu.edu}

\date{\today}

\begin{abstract} Let $G$ be a simple complex algebraic group. We prove
  that the irregularity of the adjoint connection of an irregular flat
  $G$-bundle on the formal punctured disk is always greater than or
  equal to the rank of $G$. This can be considered as a geometric
  analogue of a conjecture of Gross and Reeder. We will also show that
  the irregular connections with minimum adjoint irregularity are
  precisely the (formal) Frenkel-Gross connections. As a corollary, we establish the de Rham analogue of a conjecture of Heinloth, Ng\^o, and Yun for $G=\SL_n$. \end{abstract}

\maketitle 

\tableofcontents

\section{Introduction}

\subsection{The Gross-Reeder Conjecture} 
 Let $k$ be a $p$-adic field with
residue field $\mathfrak{f}$.  The Weil group $\cW$ of $k$ is the
subgroup of $\Gal(\bar{k}/k)$ which acts on the algebraic closure $\bar{\mathfrak{f}}$ by a
power of the Frobenius.  More explicitly, $\cW=\langle
\Fr\rangle \ltimes \cI$, where $\Fr$ is a geometric
Frobenius element and $\cI$ is the inertia group, i.e., the subgroup
of $\cW$ that acts trivially on $\bar{\mathfrak{f}}$.  The wild
inertia group $\cI_+$ is the pro-$p$-Sylow subgroup of $\cI$.

Let $G$ be a connected reductive algebraic group over $\bC$.  A
Langlands parameter (with values in $G$) is a homomorphism
$\phi:\cW\times\SL_2(\C)\to G$ such that the restriction to
$\SL_2(\C)$ is algebraic, the restriction to $\cI$ is continuous, and
$\phi(\Fr)$ is semisimple. The parameter is called discrete if the
centralizer of the image is finite; it is called inertially discrete
if there are no nonzero invariants of the action of $\phi(\cI)$ on the
Lie algebra $\mathfrak{g}$.

Note that a discrete, inertially discrete parameter $\phi$ is, in
particular, wildly ramified, i.e., the restriction of $\phi$ to
$\cI_+$ is nontrivial. A natural question that arises is what is the
minimum ``wildness'' possible for such a parameter.  One way to make
this precise is through an invariant called the (adjoint) \emph{Swan
  conductor}. Let $\Ad: G\ra \GL(\fg)$ denote the adjoint map. Then
$\Ad(\phi): \cW\times \SL_2(\C)\ra \GL(\fg)$ is a Langlands parameter
with values in a general linear group.  Given a $\GL_n$-parameter, the
Swan conductor is a canonical integer which measures how wildly
ramified it is; in particular, one can consider the Swan conductor of
$\Ad(\phi)$.  We can now state a special case of the Gross-Reeder
Conjecture \cite{GR}:

\begin{conj}[Gross-Reeder]\label{grconj} Suppose $G$ is simple. If the
  Langlands parameter $\phi:\cW\times \SL_2(\C)\ra G$ is discrete and
  inertially discrete, then the Swan conductor of $\Ad(\phi)$ is
  greater than or equal to the rank of $G$.
\end{conj}

Gross and Reeder proved their conjecture in a number of important
cases.  In particular, they verified Conjecture~\ref{grconj} under the
assumption that the residual characteristic does not divide the order
of the Weyl group.  (Reeder has recently formulated and proven a
  modified version of this conjecture~\cite{R18}.) They also analyzed the situation where the Swan
conductor equals the rank. This led them to construct \emph{simple
  wild parameters} where equality is achieved.  They also constructed
simple supercuspidal representations of a $p$-adic group with dual
group $G$ which correspond under local Langlands to simple wild
parameters.

This theory has also had important applications to the global
Langlands program.  Indeed, Heinloth, Ng\^o, and Yun used these
results to construct Kloosterman sheaves--$\ell$-adic local systems on
$\bP^1\setminus\{0,\infty\}$ whose single wildly ramified singularity
corresponds to a simple wild parameter~\cite{HNY}. These sheaves are
\emph{cohomologically rigid}, i.e., they have no infinitesimal
deformations which preserve the formal isomorphism classes at $0$ and
$\infty$.  Moreover, they provide an example of the wild ramification
case of the Langlands correspondence between $\ell$-adic local systems
and Hecke eigensheaves.

\subsection{Translation to geometry} The goal of this paper is to formulate
and prove a geometric analogue of Conjecture \ref{grconj}. This is,
therefore, a continuation of our efforts to understand wild
ramification in the geometric Langlands program \cite{BremerSagemodspace,
  BremerSageisomonodromy, BremerSagemink,
  BremerSagereg, Sageisaac, MasoudHarish, MasoudConductor, MasoudTravis, CK}.

In the geometric world, formal flat $G$-bundles play the role of
Langlands parameters, cf. the appendix of \cite{Katz}. Accordingly, we
start by reviewing their definition and some of their numerical
invariants. For more information, see \S \ref{s:connections} and
\cite{BV, Katz, BremerSagemink}.

\subsubsection{Formal flat $G$-bundles} 
Let $\cK=\C(\!(t)\!)$ denote the field of formal Laurent series. Let
$\cDt=\Spec(\cK)$ be the formal punctured disk.  A formal flat
$G$-bundle $(\cE,\n)$ is a principal $G$-bundle $\cE$ on $\cDt$
endowed with a connection $\n$ (which is automatically flat).  Upon
choosing a trivialization, the connection may be written in terms of
its matrix 
\[
\nbr_\phi \in \Om^1_{F}(\fg(\cK))
\]
 via $\n=d+\nbr_\phi$.
If one changes the trivialization by an element $g\in G(\cK)$, the
matrix changes by the \emph{gauge action}:
\begin{equation}\label{gauge}
\nbr_{g \phi} = g \cdot \nbr_\phi = \Ad(g) (\nbr_\phi) - (dg) g^{-1}.
\end{equation}
Accordingly, the set of isomorphism classes of flat $G$-bundles on
$\cDt$ is isomorphic to the quotient $\fg(\cK)\dtt/G(\cK)$, where the
loop group $G(\cK)$ acts by the gauge action.

\subsubsection{Irregular Connections} 

Recall that a flat $G$-bundle $(\cE,\phi)$ on $\cDt$ is called
\emph{regular singular} if the connection matrix has only simple poles
with respect to some trivialization; otherwise, it is called
irregular. It is well-known that irregular connections are geometric
analogues of wildly ramified Langlands parameters.  In this paper, we
will be concerned with two invariants which measure ``how irregular''
a flat $G$-bundle is: the slope and the irregularity.

\subsubsection{Slope} 

There are several equivalent definitions of the slope. The simplest to
describe (though not necessarily to compute) uses the fact that there
exists a ramified cover $\cDb=\Spec(\C(\!(u))\!)$ with $u=t^{1/b}$ and
a trivialization of the pullback bundle such that the pullback
connection is of the form
\begin{equation*}
  d+(X_{-a}u^{-a}+X_{1-a}u^{1-a}+\dots)\duu, \quad \quad X_i\in\fg,\quad
  \text{$X_{-a}$ non-nilpotent}, \quad a\geq 0.
\end{equation*}

It turns out that the quotient $a/b$ is independent of the choice of
such an expression, and one calls it the slope of
$\n$.  The
slope is positive if and only if the flat $G$-bundle is irregular, and
the smallest possible positive slope is $1/h$, where $h$ is the
Coxeter number of $G$ \cite{FG, CK, BremerSagemink}.

\subsubsection{Irregularity} 

To start with, suppose $G=\GL_n$. In this case, a flat $G$-bundle is
equivalent to a pair $(V,\n_V)$ consisting of a vector bundle on
$\cDt$ endowed with a connection.  It is a well-known result of
Turrittin \cite{Turrittin} and Levelt \cite{Levelt} that after passing
to a ramified cover $\cDb$, the pullback connection can be decomposed
as a finite direct sum
\begin{equation*}
\bigoplus (L_i\otimes M_i,\n_{L_i}\otimes \n_{M_i}),
\end{equation*}
where $(L_i,\n_{L_i})$ is rank one and $(M_i,\n_{M_i})$ is regular
singular.  Let $s_i$ denote the slope (in the sense defined above) of
the flat connection $(L_i\otimes M_i,\n_{L_i}\otimes \n_{M_i})$.  Then
the irregularity $\Irr(\n_V)$ is the sum of the slopes where each
slope $s_i$ appears with multiplicity $\dim(M_i)$. One can show that
the irregularity is a nonnegative integer that is zero if and only if
$V$ is regular singular, cf. ~\cite{Katz}.

Now, suppose $G$ is a connected reductive group. Let $(\cE,\nabla)$ be
a flat $G$-bundle on $\cDt$.  We will be interested in the
irregularity of the associated adjoint flat vector bundle$(\Ad(\cE),
\Ad(\nabla))$.  Its irregularity $\Irr(\Ad(\nabla))$ is a nonnegative
integer which is positive if and only if $(\cE,\nabla)$ is irregular.
It can be considered as the geometric analogue of the Swan conductor
of an adjoint Langlands parameter.

\subsubsection{An inequality for the adjoint irregularity} 

We are now ready to state our first main result, which is a geometric analogue of Conjecture \ref{grconj}. 

\begin{thm} \label{t:main} Let $G$ be a simple group, and let
  $(\cE,\nabla)$ be an irregular singular formal flat $G$-bundle. Then
  $\Irr(\Ad(\n)) \geq \rank(G)$.
  \end{thm} 
  
  \begin{exam} 
  This inequality is false if $G$ is not simple.  For instance,
  suppose $G=\GL_2$. Note that if $\nabla=d+A\,dt$ with $A\in \gl_2(\cK)$, then $\Ad(\nabla)=d+\Ad(A)dt$. Thus, if we take 
  \[
  \nabla=d+\on{diag}(t^{-1},
  t^{-1})\frac{dt}{t}
  \]
   then $\nabla$ is
  irregular, but $\Ad(\nabla)$ is regular singular; thus,
  $\Irr(\Ad(\nabla))=0$.
  \end{exam} 
  
  Next, we discuss when equality is achieved.

  \subsection{Formal Frenkel-Gross Connections} \label{ss:FG}

  Let $G$ be a simple complex algebraic group with Lie algebra $\fg$. 
 Let us fix a maximal torus and a Borel subgroup $H\subset
B\subset G$.  Let $\Phi$ and $\Delta=\{\a_1,\dots\a_n\}$ be the
corresponding sets of roots and simple roots.  Also, let $\a_0$ be the
highest root of $\Phi$.  We denote the root subalgebras of $\fg$ by
$\fu_\a$.  Now, choose nonzero root vectors $x_{-\a_i}\in\fu_{-\a_i}$
and $x_{\a_0}\in\fu_{\a_0}$.  Note that $N=\sum_{i=1}^n x_{-\a_i}$ is
principal nilpotent.  The global Frenkel-Gross connection associated to these root vectors is the
connection on the trivial bundle over $\bP^1$ defined by

\begin{equation}\label{eq:FG}
d+(x_{\a_0}t^{-1}+\sum_{i=1}^n x_{-\a_i})\dtt.
\end{equation} 

This connection has a regular singularity at $\infty$ and an irregular
singularity at $0$. (The connection in \cite{FG} has the locations of
the irregular and irregular singularities reversed.) It is the de Rham
analogue of the Kloosterman sheaf; in particular, it is
cohomologically rigid.

We define a formal Frenkel-Gross connection to be one which is
isomorphic to the induced formal connection at $0$ of the global
Frenkel-Gross connection:

\begin{defe}  A formal flat $G$-bundle is called a formal Frenkel-Gross
  connection if it is isomorphic to $d+(x_{\a_0}t^{-1}+\sum_{i=1}^n
  x_{-\a_i})\dtt$ for some choice of nonzero vectors
  $x_{-\a_i}\in\fu_{-\a_i}$ and $x_{\a_0}\in\fu_{\a_0}$.
\end{defe}

A priori, we have many different formal Frenkel-Gross connections as
we can multiply each $x_{-a_i}$ and $x_{\a_0}$ by nonzero complex
numbers. However, as we shall see, it is sufficient to fix one such
$(n+1)$-tuple and multiply it by nonzero scalars; see Example
\ref{Cox}.  More precisely, let $S$ be the connected centralizer of
the regular element $x_{\a_0}t^{-1}+\sum_{i=1}^n x_{-\a_i}$; it is a
\emph{Coxeter torus} (see \S\ref{ss:toral}).  The relative Weyl group
$W_S=N(S)/S$ is a cyclic group of order $h'$ dividing $h$.  We will
show that any formal Frenkel-Gross connection is isomorphic to
$d+\lambda(x_{\a_0}t^{-1}+\sum_{i=1}^n x_{-\a_i})\dtt$ for some
$\lambda\in\bC^*$.  Moreover, the connections associated to $\lambda$
and $\lambda'$ are isomorphic if and only if
$\lambda'/\lambda\in \mu_{h'}$, the $h'^{\mathrm{th}}$ roots of unity.
In other words, the set of isomorphism classes of formal Frenkel-Gross
connections is isomorphic to $\bC^*/\mu_{h'}$.  (Of course, this space
is homeomorphic to $\bC^*$.  To get a set of representatives indexed
by $\bC^*$, one fixes the $x_{-\a_i}$'s and multiplies $x_{\a_0}$ by a
constant.)

We are now ready to state the companion result to Theorem \ref{t:main}: 
\begin{thm} \label{t:main2} Let $G$ be a simple group, and let
  $(\cE,\nabla)$ be an  formal flat $G$-bundle. Then
 the following are equivalent:
\begin{enumerate}
\item $\Irr(\Ad(\n))=\rank(G)$;
\item $\slope(\n)=\frac{1}{h}$; 
\item $\nabla$ is a formal Frenkel-Gross connection.
\end{enumerate}
\end{thm} 

We remark that the proofs of the two theorems use quite different
methods.  The proof of Theorem \ref{t:main} makes use of the classical
Levelt-Turritin theory of Jordan forms for formal flat $G$-bundles.
This allows us to translate the desired inequality into a statement
about eigenvectors of elements of finite real reflection groups; see
also \cite{Masoud}. We check this explicitly for each type.  The
statement makes sense in the context of complex reflection groups as
well, and we conjecture that it holds in general.  In contrast, our
proof of Theorem \ref{t:main2} requires non-classical methods.
Indeed, the proof uses the full power of the geometric theory of
fundamental and regular strata for formal flat $G$-bundles developed
in \cite{BremerSagemink,BremerSagereg}.

\subsection{A de Rham analogue of a conjecture of Heinloth, Ng\^o, and Yun} 
In \cite[\S 7.1]{HNY}, the authors have conjectured that the
Kloosterman sheaves are the only $\ell$-adic local systems on
$\mathbb{P}^1\setminus\{0,\infty\}$ with certain prescribed behaviour
at the marked points. They also gave an analogous construction of
Kloosterman $\mathrm{D}$-modules and conjectured that they were the
same as global Frenkel-Gross connections.  This was subsequently
proved by Zhu~\cite{Zhu}.  Accordingly, one can translate the
conjecture of Heinloth, Ng\^o, and Yun on Kloosterman sheaves into the de Rham setting as follows:

\begin{conj}[De Rham analogue of Conjecture 7.2 of \cite{HNY}] \label{c:HNY}
Let $\nabla$ be a $G$-connection on $\mathbb{P}^1$ which is regular away from $\{0,\infty\}$ and satisfies 
\begin{itemize} 
\item $\nabla_\infty$ is regular singular  with principal unipotent monodromy.
\item $\nabla_0$ is  irregular with irregularity equal to $r=\rank(G)$.
\end{itemize} 
Then $\nabla$ is isomorphic to a global Frenkel-Gross connection. 
\end{conj}

\begin{thm} Conjecture \ref{c:HNY} holds for $G=\SL_n$. 
\end{thm} 

\begin{proof} By Theorem \ref{t:main2}, there exists a (global)
  Frenkel-Gross $\SL_n$-connection $\nabla^\FG$ such that
  $\nabla_0\simeq \nabla_0^\mathrm{FG}$. Moreover, since $\nabla$ and
  $\nabla^\FG$ are both regular singular at $\infty$ and have the same
  monodromy, we have $\nabla_\infty\simeq \nabla^\FG_\infty$.
  According to \cite{FG}, the Frenkel-Gross connection $\nabla^\FG$ is
  cohomologically rigid. In fact, Katz's rigidity theorem
  \cite[Theorem 3.7.3]{Katzbook} shows that $\nabla^\FG$ is also
  cohomologically rigid when viewed as a $\GL_n$-connection.  To apply
  this theorem, we need to observe that $\nabla^\FG$ is irreducible as
  a $\GL_n$-connection and that $\chi(\Gm, \nabla^\FG )=-1$, where
  $\chi$ is the Euler characteristic.  The first fact is easy; it is
  even true of the localization $\nabla_0^\FG$.  The second follows
  from Deligne's formula for the Euler characteristic~\cite[\S 6.19,
  (6.21.1)]{Deligne}:

  \begin{equation*}\chi(\Gm, \nabla^\FG )=\chi(U)\rk(\nabla^\FG
  )-\Irr(\nabla_0^\FG)-\Irr(\nabla_\infty^\FG)=0-1-0=-1.
\end{equation*} 

  By the main result of \cite{BE}, $\nabla^\FG$ is physically rigid
  when viewed as a $\GL_n$-connection. This implies that there exists
  $g\in \GL_n(\bC[t,t^{-1}])$ such that $g.\nabla=\nabla^\FG$, where
  the action is gauge transformation as in \eqref{gauge}. It remains
  to show that $g$ can be chosen to be in $\SL_n$.  To this end, let
  us write
\[
g=z^{-1}g',\quad \quad g'\in \SL_n(\bC[t,t^{-1}]),\quad z=\mathrm{diag}(\det(g)^{-1}, 1, 1, ... , 1).
\]
 Then, we have 
\[
g'.\nabla=z.\nabla^\FG=\nabla^\FG-(dz)z^{-1}\implies g'.\nabla-\nabla^\FG=-(dz)z^{-1}. 
\]
Now, observe that the matrix of the connection $g'.\nabla-\nabla^\FG$
is traceless. This means that the $(dz)z^{-1}$ is traceless and
diagonal, hence zero.   Thus, $g'.\nabla=\nabla^\FG$, so $\nabla$ and $\nabla^\FG$ are equivalent as $\SL_n$-connections.
\end{proof} 

\subsection{Further directions}

We observe that our characterization of the formal Frenkel-Gross
connection makes explicit the notion that it should be viewed as the
geometric version of the simple wild parameters of Gross and Reeder.
This perspective also suggests a potential generalization of the
results of this paper.  Reeder and Yu have constructed
\emph{epipelagic representations} of $p$-adic groups, certain
supercuspidal representations which generalize Gross and Reeder's
simple supercuspidals~\cite{RY}.  This theory has been used by Yun in
his construction of \emph{generalized Kloosterman
  sheaves}~\cite{Yun}.  Reeder has very recently shown that
the corresponding epipelagic Langlands parameters also are the
parameters for which equality holds in a certain inequality involving
the Swan conductor~\cite{R18}.

The theory of regular strata suggests that the geometric analogue of
epipelagic parameters are elliptic toral connections with minimal
(positive) slope.  In fact, regular strata can also be used to
construct de Rham analogues of generalized Kloosterman sheaves, whose
irregular singularity is a formal connection of such a type.  (Yun has
also accomplished this through his notion of
\emph{$\theta$-connections}, cf. \cite{Chen}.)  We expect that these
formal connections can be characterized as those connections for which
equality holds in a more complicated inequality involving the adjoint
irregularity.  We will consider this issue in a future paper.

\subsection{Organization and Notation} 

In \S \ref{s:connections}, we review the Jordan
form (aka the Levelt-Turrittin form) of a formal flat $G$-bundles.
Using this, we give alternative definitions of slope and irregularity.
Properties of formal Frenkel-Gross connections are established in \S
\ref{s:fg}. We use the theory of fundamental strata for formal
connections \cite{BremerSagereg, BremerSagemink} to prove that every
connection with slope $1/h$ is a formal Frenkel-Gross connection. In
\S \ref{s:Weyl}, we prove a result about Weyl groups which will be
crucial for our main theorems. Finally, we prove Theorem \ref{t:main}
and \ref{t:main2} in \S \ref{s:Proof}.

Throughout the paper, $G$ denotes a connected complex reductive group
with Lie algebra $\fg$; unless otherwise specified, we assume that $G$
is simple.  We fix a Borel subgroup $B$ with maximal torus $H$ and
unipotent radical $N$; the corresponding Lie algebras are denoted by
$\fb$, $\fh$, and $\fn$.  As in \S\ref{ss:FG}, $\Phi$ and
$\Delta=\{\a_1,\dots\a_n\}$ are the corresponding sets of roots and
simple roots, and $W$ is the Weyl group. When $\Phi$ is irreducible,
we let $\a_0$ be the highest root of $\Phi$.  We denote the root
subalgebras of $\fg$ by $\fu_\a$.

\subsection{Acknowledgements} We would like to thank Chris Bremer,
Tsao-Hsien Chen, and Dick Gross for helpful conversations regarding
this project.  The first author would like to thank Jim Humphreys, Gus
Lehrer, Peter McNamara, Chul-hee Lee, Daniel Sage, and Ole Warnaar for
helpful discussions regarding reflection groups. MK is supported by
the Australia Research Council DECRA Fellowship. DS is supported by an
NSF Grant and a Simons Collaboration Grant.

\section{Jordan decomposition for formal connections}\label{s:connections}
In this section, we recall some basic facts regarding formal flat
$G$-bundles. Here, $G$ is not assumed to be simple.

\subsection{Jordan decomposition} For each positive integer $b$,
let $\cK_b=\bC(\!(t^{\frac{1}{b}})\!)$ denote the unique finite extension
of $\cK$ of degree $b$ and let $\cDt_b$ denote the corresponding
punctured disk. Let $\pi_b: \cDt_b\ra \cDt$ denote the canonical
covering map. Given a flat $G$-bundle $(\cE,\nabla)$, we denote by
$(\pi_b^*\cE, \pi_b^*\nabla)$ its pullback to $\cDt_b$. For ease of
notation, we sometimes use the substitution $u=t^{\frac{1}{b}}$.

\begin{thm}\label{t:Jordan} Let $(\cE,\nabla)$ be a formal flat $G$-bundle. Then there exists a positive integer $b$ and a trivialization of $\pi_b^*\cE$ in which $\pi_b^*\nabla$ can be written as 
 \begin{equation}\label{eq:Jordan}
\pi_b^*\nabla=d+(h+n)\frac{du}{u}, \quad \quad h\in \fh[u^{-1}], \quad n\in \fn(\bC),  
\end{equation}
and $h$ and $n$ commute. Moreover, the pair $h$ and $n$ satisfying the above properties is unique. 
\end{thm} 

\begin{defe} 
  We call $d+(h+n)\frac{du}{u}$ the \emph{Jordan form} of the formal
  flat connection $(\cE,\nabla)$.
\end{defe} 

For $G=\GL_n$, the existence of the trivialization with the properties
in the above theorem was first proved by Turrittin \cite{Turrittin}.
Subsequently, Levelt proved uniqueness \cite{Levelt}.  Babbitt and
Varadarajan \cite{BV} have given an alternative proof of this fact. In
addition, following a suggestion of Deligne, they showed that the
above theorem also holds for $G$ an arbitrary reductive group \cite[\S
9]{BV}.

 \subsection{Slope} Using the Jordan decomposition, we can give an
 alternative definition of the slope of a formal flat $G$-bundle.  If
 $z$ is a Laurent series, we let $\ord_\pole(z)\in \bZ_{\geq 0}$
 denote the order of the pole of $z$. Clearly, $z$ is a power series
 (i.e., has no singularity) if and only if $\ord_\pole(z)=0$.

Let $(\cE,\nabla)$ be a formal flat $G$-bundle and let
  $d+(h+n)\frac{du}{u}$ denote its Jordan form, where $u=t^{\frac{1}{b}}$ for some $b\in \bZ_{\geq 1}$.
  
  \begin{defe}  The non-negative
  rational number $s=\max\{0,\frac{\ord_{\pole}(h)}{b}\}$ is called
  the \emph{slope} of $\nabla$.
\end{defe}

It is immediate that this definition coincides with the one given in
the introduction when $h\ne 0$, since the leading term of $h+n$ is
evidently non-nilpotent.  If $h=0$, then $\n$ is regular singular, and
both definitions give $0$.  We will also need an equivalent definition
of the slope given in terms of fundamental
strata~\cite{BremerSagemink}; this will be discussed in \S\ref{s:fg}.  For yet
another equivalent definition, see \cite[\S 2]{CK}.

\subsection{Irregularity} Let $G=\GL_n$, and let $B$ (resp. $H$)
denote the upper triangular (resp. diagonal) matrices.  Let
$(\cE,\nabla)$ be a formal flat $G$-bundle (equivalently, a vector
bundle on $\cDt$ equipped with a connection). Let
$d+(h+n)\frac{du}{u}$ denote its Jordan form with respect to the above
choice of $B$ and $H$. (Recall that $u=t^{\frac{1}{b}}$ for some
$b\in \bZ_{\geq 1}$.)  Since $h$ is diagonal, we can write
  \[
  h=\on{diag}(h_1,\cdots, h_n),\quad \quad h_i\in \bC[u^{-1}].
  \]

\begin{defe}  The irregularity of $(\cE,\nabla)$ is defined by
\[
 \Irr(\n)=\sum_{i=1}^n \{\max\{0,\frac{\ord_{\pole}(h_i)}{b}\}\}.
 \]
\end{defe} 

One can show that $\Irr(\nabla)\in \bZ_{\ge 0}$ and that this integer coincides with the one defined in the introduction.

\section{Irregular connections with minimum
  slope}\label{s:fg}

In \S\ref{ss:FG}, we introduced the notion
of a formal Frenkel-Gross connection.  In this section, we will
characterize these connections as the irregular flat $G$-bundles on
$\cDt$ with minimum slope, i.e., with slope $1/h$, where $h$ is the
Coxeter number of $G$. 

Recall from \S\ref{ss:FG} that the connected centralizer of the matrix
of a Frenkel-Gross connection is a maximal torus of $G(\cK)$ called a
Coxeter torus; its relative Weyl group is a cyclic group of order
$h'$ dividing $h$.

\begin{prop} \label{p:FG}
For a flat $G$-bundle $\nabla=(\cE,\nabla)$ the
following are equivalent: 
\begin{enumerate} 
\item[(i)] $\nabla$ is  isomorphic to a Frenkel-Gross connection;
\item[(ii)] The slope of $\nabla$ equals $\frac{1}{h}$. 
\end{enumerate} 
Moreover, the set of isomorphism classes of
such flat $G$-bundles is in bijection with $\C^*/\mu_{h'}$.  Finally, for each such connection, $\Irr(\Ad(\n))=\rank(G)$.
\end{prop}

In order to prove the proposition, we will need to recall some of the
geometric theory of fundamental strata from ~\cite{BremerSagemink,
  BremerSagereg}.

\subsection{Fundamental strata}

 Let $\cB$ be the Bruhat-Tits building of $G$; it is
a simplicial complex whose facets are in bijective correspondence with
the parahoric subgroups of the loop group $G(\cK)$.  The standard apartment $\cA$
associated to the split maximal torus $H(\cK)$ is an affine space
isomorphic to $X_*(H)\otimes_\bZ \bR$.  Given $x\in\cB$, we denote by
$G(\cK)_x$ (resp. $\fg(\cK)_x$) the parahoric subgroup (resp.
subalgebra) corresponding to the facet containing $x$.

For any $x\in\cB$, the Moy-Prasad filtration associated to $x$ is a
decreasing $\bR$-filtration 
\[
\{\fg(\cK)_{x,r}\mid
r\in\bR\}
\] of $\fg(\cK)$ by $\C[[t]]$-lattices.  The filtration satisfies
$\fg(\cK)_{x,0}=\fg(\cK)_x$ and is periodic in the sense that
$\fg(\cK)_{x,r+1}=t\fg(\cK)_{x,r}$.  Moreover, if we set
$\fg(\cK)_{x,r+}=\bigcup_{s>r}\fg(\cK)_{x,s}$, then the set of $r$ for
which $\fg(\cK)_{x,r}\ne\fg(\cK)_{x,r+}$ is a discrete subset of
$\bR$.

For our purposes, it will suffice to give the explicit definition for
$x\in\cA$.  In this case, the filtration is determined by a grading on
$\fg(\bC[t,t^{-1}])$, with the graded subspaces given by
\begin{equation*}
\fg(\cK)_{x}(r)=\begin{cases} \fh t^r\oplus
  \bigoplus\limits_{\a(x)+m=r}\fu_\a t^m, &\text{if } r\in\bZ\\\\
\bigoplus\limits_{\a(x)+m=r}\fu_\a t^m,&\text{otherwise.}
\end{cases}
\end{equation*}

Let $\kappa$ be the Killing form for $\fg$.  Any element $X\in
\fg(\cK)$ gives rise to a continuous $\C$-linear functional on
$\fg(\cK)$ via $Y\mapsto \Res\kappa(Y,X)\dtt$.  This identification
induces an isomorphism
\begin{equation*} 
(\fg(\cK)_{x,r}/\fg(\cK)_{x,r+})^\vee\cong \fg(\cK)_{x,-r}/\fg(\cK)_{x,-r+}.
\end{equation*}

The leading term of a connection with respect to a Moy-Prasad
filtration is given in terms of $G$-\emph{strata}.
A $G$-stratum of depth $r$ is a triple $(x,r,\b)$ with $x\in\cB$,
$r\ge 0$, and $\b\in(\fg(\cK)_{x,r}/\fg(\cK)_{x,r+})^\vee$.  Any
element of the corresponding $\fg(\cK)_{x,-r+}$-coset is called a
representative of $\b$.  If $x\in\cA$, there is a unique homogeneous
representative $\b^\flat\in \fg(\cK)_{x}(-r)$.  The stratum is called
\emph{fundamental} if every representative is non-nilpotent.  When
$x\in\cA$, it suffices to check that $\b^\flat$ is non-nilpotent.

\begin{defe} If $x\in\cA\cong\fh_\bR$, we say that $(\cE,\n)$ contains the stratum $(x,r,\b)$
with respect to the trivialization $\phi$ if $[\nabla]_\phi-x\dtt\in
\fg(\cK)_{x,-r}\dtt$ and is a representative for $\b$.  (See
\cite{BremerSagemink} for the general definition.)
\end{defe}

We recall some of the basic facts about the relationship between flat
$G$-bundles and strata.   The following theorem and
  Theorem~\ref{formaltypes} hold for reductive $G$.

\begin{thm}[{\cite[Theorem 2.14]{BremerSagemink}}]
 Every flat $G$-bundle $(\cE, \n)$ contains a fundamental
  stratum $(x,r,\b)$, where $x$ is in the fundamental alcove
  $C\subset\cA$ and $r\in\bQ$; the
  depth $r$ is positive if and only if $(\cE, \n)$ is irregular
  singular.  Moreover, the following statements hold.
\begin{enumerate}
\item If $(\cE, \n)$ contains the stratum $(y,r', \b')$, then
$r' \ge r$.  
\item  If $(\cE,\n)$ is irregular singular, a stratum
  $(y,r', \b')$ contained in $(\cE, \n)$ is fundamental if and only if
  $r' = r$.
\end{enumerate}
\end{thm}
As a consequence, one can define the slope of $\n$ as the depth of any
fundamental stratum it contains.

As an example, let $x_I$ be the barycenter of the fundamental alcove
in $\cA$, which corresponds to the standard Iwahori subgroup.  It is
immediate from the definition that
$\fg(\cK)_{x_I}(-1/h)=t^{-1}\fu_{\a_0}\oplus\bigoplus_{i=1}^n
\fu_{-\a_i}$.  One now sees that a Frenkel-Gross connection contains a
stratum of the form $(x_I,\frac{1}{h},\b)$.  This stratum is
fundamental; in fact, $\b^\flat$ is regular semisimple. In particular, the slope of a formal Frenkel-Gross connection is $\frac{1}{h}$. 
We note that Frenkel and Gross derived this directly from the definition of slope given in the introduction.

\subsection{Regular strata and toral flat $G$-bundles}\label{ss:toral}

We will also need some results on flat $G$-bundles which contain a
\emph{regular stratum}, a kind of stratum that satisfies a graded
version of regular semisimplicity.  For convenience, we will only
describe the theory for strata based at points in $\cA$.

Let $S\subset G(\cK)$ be a (in general, non-split) maximal torus, and
let $\fs \subset \fg(\cK)$ be the associated Cartan subalgebra.  We
denote the unique Moy-Prasad filtration on $\fs$ by $\{\fs_r\}$.  More
explicitly, we first observe that if $S$ is split, then this is just
the usual degree filtration.  In the general case, if $\cK_b$ is a
splitting field for $S$, then $\fs_r$ consists of the Galois fixed
points of $\fs(\cK_b)_r$.  Note that $\fs_r\ne\fs_{r+}$ implies that
$r\in\bZ\frac{1}{b}$.

A point $x \in \cA$ is called \emph{compatible} with $\fs$ if
$\fs_r=\fg(\cK)_{x,r}\cap \fs$ for all $r\in\bR$.

\begin{defe} A fundamental stratum
  $(x,r,\beta)$ with $x\in\cA$ and $r>0$ is an \emph{$S$-regular stratum} if $x$ is compatible with
  $S$ and $S$ equals the connected centralizer of $\b^\flat$.
\end{defe}
In fact, every representative of $\b$ will be regular semisimple with
connected centralizer a conjugate of $S$.

\begin{defe}  If $(\cE,\n)$ contains the $S$-regular stratum
  $(x,r,\b)$, we say that $(\cE,\n)$ is \emph{$S$-toral}.
\end{defe}

Recall that the conjugacy classes of maximal tori in $G(\cK)$ are in
one-to-one correspondence with conjugacy classes in the Weyl group $W$
\cite{KL88}.  It turns out that there exists an $S$-toral flat
$G$-bundle of slope $r$ if and only if $S$ corresponds to a regular
conjugacy class of $W$ and $e^{2\pi i r}$ is a regular eigenvalue for
this class~\cite[Corollary 4.10]{BremerSagereg}.  Equivalently,
$\fs_{-r}\setminus\fs_{-r+}$ contains a regular semisimple element.
For example, a Frenkel-Gross connection is $S$-toral for $S$ a Coxeter
torus, i.e., a maximal torus corresponding to the Coxeter conjugacy
class in $W$.  (One way to see this is that $e^{2\pi i/h}$ is a
regular eigenvalue for Coxeter elements and for no other elements of
$W$.)  Moreover, since the regular eigenvalues of a Coxeter element
are the primitive $h^\mathrm{th}$ roots of $1$, the possible slopes for
$S$-toral flat $G$-bundles are $m/h$ with $m>0$ relatively prime to
$h$.

An important feature of $S$-toral flat $G$-bundles is that they can be
``diagonalized'' into $\fs$.  To be more precise, suppose that there
exists an $S$-regular stratum of depth $r$.  The filtration on $\fs$
can be defined in terms of a grading, whose graded pieces we denote by
$\fs(r)$.  Let $\cA(S,r)$ be the open subset of
$\bigoplus_{j\in[-r,0]}\fs(j)$ whose leading component (i.e., the
component in $\fs(-r)$) is regular semisimple.  This is called the set
of \emph{$S$-formal types} of depth $r$.  Let $\Waff_S=N(S)/S_0$ be
the relative affine Weyl group of $S$; it is
the semidirect product of the relative Weyl group $W_S$ and the free
abelian group $S/S_0$.  The group $\Waff_S$ acts on $\cA(S,r)$.  The
action of $W_S$ is the restriction of the obvious linear action while
$S/S_0$ acts by translations on $\fs(0)$.
\begin{thm}{\cite[Corollary 5.14]{BremerSagereg}}\label{formaltypes} If $(\cE,\n)$ is $S$-toral of
  slope $r$, then $\n$ is gauge-equivalent to a connection with matrix
  in $\cA(S,r)\dtt$.  Moreover, the moduli space of $S$-toral flat
  $G$-bundles of slope $r$ is given by $\cA(S,r)/\Waff_S$.
\end{thm}

\begin{exam}\label{Cox}
  Let $S$ be a Coxeter maximal torus.  After conjugating, we may
  assume that it is the connected centralizer of the matrix
  $\zeta\dtt$ of a Frenkel-Gross connection with
  $\zeta=x_{\a_0}t^{-1}+\sum_{i=1}^n x_{-\a_i}\in
  t^{-1}\fu_{\a_0}\oplus\bigoplus_{i=1}^n \fu_{-\a_i}$.  In this case,
  $x_I$ is \emph{graded compatible} with $\fs$, i.e.,
  $\fs(r)=\fg_{x_I}(r)\cap\fs$ for all $r$. It is easy to see that
  $S$ is elliptic.  Indeed, an element of
  $\fs(0)\subset\fg_{x_I}(0)=\fh$ would commute with the principal
  nilpotent element $N=\sum_{i=1}^n x_{-\a_i}$.  However, since $G$ is
  simple, 
  $\fz(N)$ is a subset of $\bar{\fn}$, the nilpotent radical of the Borel
  subalgebra opposite to $\fb$, so $\fs(0)=\{0\}$.  This means that the
  action of $S/S_0$ on $\cA(S,\frac{m}{h})$ is trivial, so the moduli
  space of $S$-toral flat $G$-bundles of slope $m/h$ is just
  $\cA(S,\frac{m}{h})/W_S$.

 When $m=1$, we may be entirely explicit.  First, we show that
 $\fs(-1/h)$ is one-dimensional.  Suppose that
 $y_{\a_0}t^{-1}+N'\in\fs(-1/h)\subset\fg_{x_I}(-1/h)$; here,
 $y_{\a_0}\in\fu_{\a_0}$ and $N'\in\fg_{x_I}(-1/h)\cap\fn$.  Since $N$
 is regular nilpotent and $[N,N']=0$, $N'=\lambda N$ for some
 $\lambda\in\bC$.  It follows that $[N,\lambda x_{\a_0}-y_{\a_0}]=0$,
 so $\fz(N)\subset\bar{\fn}$ implies $y_{\a_0}=\lambda x_{\a_0}$.
 Next, $W_S$ is
  isomorphic to a subgroup of the centralizer of a Coxeter element in
  $W$~\cite[Proposition 5.9]{BremerSagereg}, so $W_S$ is a cyclic
  group of order $h'|h$.  Since nonzero elements of $\fs(-1/h)$ are
  regular semisimple, $W_S$ acts freely on it.  Thus, the moduli space
 $\cA(S,\frac{m}{h})/W_S$ is isomorphic to $\bC^*/\mu_{h'}$.
\end{exam}

\subsection{Proof of Proposition~\ref{p:FG}} 

We will begin the proof of the proposition with two lemmas.

\begin{lem} If $(\cE,\n)$ is an $S$-toral flat $G$-bundle of slope
  $r$, then $\Irr(\Ad(\n))=r|\Phi|$.
  In particular, a Frenkel-Gross
  connection has irregularity $r=\rank(G)$.
\end{lem}

A consequence of this lemma is that for a toral flat $G$-bundle of
slope $r$,  $r|\Phi|$ is an integer.

\begin{proof}  Suppose that $S$ splits over the degree $b$ extension
  $\cK_b$ with uniformizer $u$ such that $u^b=t$.  The pullback
  connection $\n'$ is toral for a split maximal torus.  Accordingly,
  we can choose a trivialization for which $[\n']=X\dtt$, where $X\in\fh(\cK_b)$ with regular semisimple leading term;
  moreover, for each root $\alpha$, $\alpha([X])$ has valuation
  $-rb$.  Thus, the adjoint connection of $\n'$ is the
  direct sum of $|\Phi|$ flat line bundles of slope $rb$ and $n$
  trivial flat line bundles.  It follows that the irregularity of $\Ad(\n)$
  is $|\Phi|rb/b$ as desired.

  If $\n$ is a Frenkel-Gross connection, it is $S$-toral with $S$ a
  Coxeter maximal torus and has slope $1/h$.  Since $|\Phi|/h=n$,
  $\Irr(\Ad(\n))=n$.
\end{proof}

Now, let $\n$ be a flat $G$-bundle of slope $1/h$.  We want to show
that $\n$ must be a formal Frenkel-Gross connection. The fact that
the slope of $\n$ equals $1/h$ means that $\n$ contains a fundamental
stratum $(x,1/h,\b)$ for some $x\in\cB$ with respect to some
trivialization.  By equivariance, we can assume that $x$ lies in the
fundamental alcove $C$ corresponding to the standard Iwahori subgroup
$I$.  Let $x_I$ be the barycenter of $C$.

\begin{lem} The barycenter $x_I$ is the unique point $x\in C$ for which
  there exists a fundamental stratum of the form $(x,1/h,\b)$.
\end{lem}
\begin{proof}  Suppose there is a fundamental stratum of depth $1/h$
  based at $x$.  Since $x\in C$, $\a_i(x)\ge 0$ for $1\le i\le n$ and
  $\a_0(x)\le 1$.  This implies that  if $\a$ is positive
  (resp. negative), then $\a(x)\in[0,1]$ (resp.  $\a(x)\in[-1,0]$.
  As a result, 
\begin{equation*}\fg_x(-1/h)=\bigoplus_{\{\a<0\mid
      \a(x)=-1/h\}} \fu_\a \oplus \bigoplus_{\{\a>0\mid
      \a(x)=1-1/h\}} t^{-1}\fu_\a.
  \end{equation*}

Let $I=\{i\in [1,n]\mid \a_i(x)\le 1/h\}$, and let $J=I^c$.  Let
$\fp_I$ be the standard parabolic subalgebra generated by $\fb$ and
those $\fu_{\a_i}$ with $i\in I$.  We denote its standard Levi decomposition
by $\fp_I=\fl_I\oplus \fn_I$.  

Since $\fg_x(-1/h)$ contains a non-nilpotent element, there must exist
a positive root $\a$ with $\a(x)=1-1/h$.  (Otherwise, $\fg_x(-1/h)$ is
contained in $\bar{\fn}$.)  Since any positive root is the sum of
at most $h-1$ simple roots, either the decomposition of $\a$ into
simple roots involves $\a_j$ for $j\in J$ or else $J$ is empty,
$\a_i(x)=1/h$ for all $i$, and $\a=\a_0$.  The second case is just
$x=x_I$.

It remains to show that $J$ cannot be nonempty.  If not, then we see
that $\fg_x(-1/h)\subset (\fl\cap\bar{\fn})\oplus \fn_I$.  However, if
$X\in\fp_I$ is the sum of a nilpotent element of $\fl_I$ and an
element in $\fn_I$, it is nilpotent.  Thus, every element of
$\fg_x(1/h)$ is nilpotent, a contradiction.
\end{proof}

\begin{rem} A similar argument gives another proof that $1/h$ is the
  smallest possible slope of an irregular singular flat $G$-bundle.
  Indeed, if this were false, then there would exist a fundamental
  stratum $(x,r,\b)$ with $x\in C$ and $0<r<1/h$. Setting $I=\{i\in
  [1,n]\mid \a_i(x)\le r\}$, we obtain that for any $\a>0$ coming from
  $\fl_I$, $\a(x)\le (h-1)r<1-1/h<1-r$.  It follows that
  $\fg_x(-r)\subset (\fl\cap\bar{\fn})\oplus \fn_I$ consists entirely
  of nilpotent elements, a contradiction.
\end{rem}

We now know that $\n$ contains a fundamental stratum $(x_I,1/h,\b)$ so
that the leading term of $[\n]$ with respect to the $x_I$ filtration
has the form $(t^{-1}y_0 +\sum_{i=1}^n y_{i})\dtt$ with $y_{0}\in\fu_{\a_0}$ and
$y_{i}\in\fu_{-\a_i}$ for $i\ge 1$.  In order for this element to be
non-nilpotent, every $y_i$ must be nonzero.  Indeed, if some $y_i$
equals $0$, then the remaining $n$ roots are a base for a maximal rank
reductive subalgebra of $\fg$, so that the  leading term is
nilpotent, a contradiction.

We thus have each $y_i$ nonzero.  It follows that this leading term is
regular semisimple with centralizer a Coxeter maximal torus $S$ with
Lie algebra $\fs$.  The unique Moy-Prasad filtration on $\fs$ is
induced by a grading with degrees in $\frac{1}{h}\bZ$.  By
Theorem~\ref{formaltypes}, after applying a gauge change, one can
assume that $[\n]\in(\fs(-1/h)\oplus \fs(0))\dtt$.  However, since $G$
is simple and $S$ is elliptic, $\fs(0)=\{0\}$.  Hence,
$\n=d+(t^{-1}x_0 +\sum_{i=1}^n x_i)\dtt$ in this trivialization.
Since $t^{-1}x_0 +\sum_{i=1}^n x_i\in\fs(-1/h)$ is regular semisimple,
each $x_i$ is nonzero, i.e., $\n$ is a Frenkel-Gross flat $G$-bundle.
In fact, as shown in Example~\ref{Cox}, $\fs(-1/h)$ is
one-dimensional, and the cyclic group $W_S\cong\mu_{h'}$ acts freely
on $\fs(-1/h)\setminus\{0\}$.  Thus, the set of isomorphism classes of Frenkel-Gross
connection is isomorphic to $\bC^*/\mu_{h'}$.  This concludes the
proof of the proposition.\qed

\section{A key result about Weyl groups} \label{s:Weyl}

For the remainder of the paper, let $G$ be a simple complex algebraic
group of rank $r$. Recall that $|\Phi|=hr$.

\subsection{Statement of the result} 

For $x\in \fh$, define a non-negative integer  
\begin{equation} \label{eq:Nx}
N(x):=|\{ \alpha \in \Phi \, | \, \a(x)\neq 0\}|.
\end{equation} 
The stabilizer $W_x$ of $x$ is a parabolic subgroup generated by the
set $\Phi_x$ for which $\a(x)=0$, so that $N(x)=|W|-|W_x|$.

For each positive integer $b$, let $V(b)\subseteq \fh$ denote the
union of all the eigenspaces of elements of $W$ with eigenvalue a primitive
$b^\mathrm{th}$ root of unity, i.e. 
\[
V(b)=\{ x\in \fh \, | \, w x=\zeta x \,\, \textrm{for some $w\in W$ and some primitive $b^\mathrm{th}$ root of unity $\zeta$}\}.
\]
Let $d_1,\dots,d_m$ be the degrees of $W$, and choose $Q_1,\dots,
Q_m\in\bC[\fh]^W$ homogeneous polynomials with $\deg(Q_i)=d_i$ that
generate $\bC[\fh]^W$.  It is a result of
Springer~\cite[Proposition 3.2]{Springer} that 
\begin{equation}\label{vbq} V(b)=\bigcap_{b\nmid d_i}\{x\mid Q(x)=0\}.
\end{equation}
In particular, $V(b)=\{0\}$ unless $b$ divides an exponent
of $W$, so $V(b)$ nonzero implies that $b\le h$.

The following result will play a key role for us:

\begin{thm}\label{t:reflection}  Let $b$ be a positive integer, and suppose  $x\in V(b)\setminus\{0\}$. Then
  $N(x) \geq br$.  Moreover, equality is achieved if and only if
  $b=h$, in which case $x$ is a regular eigenvector for a Coxeter element of
  $W$.
\end{thm} 

For more details about this result, including its relation with
previous results on eigenvalues of Coxeter elements, we refer the
reader to \cite{Masoud}.

We remark that the case $b=h$ is a theorem of Kostant~\cite[\S
9]{Kostant}, so in what follows, we assume that $b<h$.  We now discuss
the proof, which uses a case by case analysis.

\subsection{Proof for classical groups}

The proof in the classical types proceeds as follows. First, we rule
out small values of $b$. Next, we note that $x$ must be a root of a
certain class of polynomials. Finally, we analyze the roots of these
polynomials and show that their stabilizers in $W$ can never be ``too
big''. To illustrate this, we give the complete proof for $G=\SL_n$.
The adaptation to other classical groups is straightforward.

Let $V$ be the subspace of the $\bR^n$ consisting
of $n$-tuples $(x_1,\cdots, x_n)$ satisfying $\sum x_i=0$. Here, the
Weyl group is $S_n$.  It will be convenient to take the coefficients of
the characteristic polynomial as the invariant polynomials on $\fg$
and on $\fh$. Up to sign, these are the elementary symmetric
polynomials: if $x=\on{diag}(x_1,\cdots, x_n) \in \fh\cong V_\bC$, then the
characteristic polynomial of $x$ is
\[
P(X)=\sum c_i X^i,
\]
where $c_1=-\sum x_i=0, \cdots, c_n=(-1)^n x_1\cdots x_n$.

\begin{lem} Suppose $x\in V(b)$. Then, there exists $ 1\leq k\leq \frac{n}{b}$ and
complex numbers $a_i$ such that the $x_i$'s are the roots of the
polynomial
\begin{equation}
P(X):=X^n+a_1 X^{n-b} + \cdots + a_k X^{n-bk}, \quad \quad a_k\neq 0.
\end{equation} 
\end{lem} 

\begin{proof} 
 Indeed, suppose $x=(x_1,\cdots,
x_n)\in V(b)$.  Then by \eqref{vbq}, $c_i=0$  for all $i$ with $b\nmid i$. Now take $a_i=c_{n-bi}$. 
\end{proof}

Let $x\in V_\bC$ be a non-zero element, so that $W_x$ is a
proper parabolic subgroup. It is easy to check that the proper
parabolic of $W$ with the largest number of roots is isomorphic to
$S_{n-1}$. Thus, for all non-zero $x\in V_\bC$, we have
\[
|\Phi|-|\Phi_x|\geq n(n-1)-(n-1)(n-2) = 2(n-1)>(n-1).
\]
Therefore, the theorem is evident for $b=1$. 

Henceforth, we assume $1<b<n$, so $n\ge 3$.  Let $P(X)$ denote the polynomial in the above lemma. Let $\gamma_1,\cdots, \gamma_k$ denote the roots of the polynomial
\[
Q(Y):=Y^k + a_1Y^{k-1} + \cdots + a_k.
\]
Since $a_k\neq 0$, we have that $\gamma_i\neq 0$ for all $i$. Let
$\zeta$ be a primitive $b^\mathrm{th}$ root of unity. Then, the roots
of $P(X)$ equal
\[
\zeta^i \gamma_j, \quad \quad i\in \{1,2,\cdots, b\},\quad j\in
\{1,2,\cdots, k\}
\]
together with $n-kb$ copies of $0$.

Note that the largest possible stabilizer for $x$ (for fixed $b$) is
achieved if and only if $\zeta^{i_1} \gamma_1=\zeta^{i_2} \gamma_2 =
\cdots = \zeta^{i_{k}} \gamma_k$ for some integers $i_1,\cdots, i_k$.
In this case,
\[
W_x \simeq (S_k)^{b} \times S_{n-bk}.
\]

Thus, for every non-zero $x\in V_\bC$, we have 
\[
\begin{aligned} 
  N(x) = |\Phi| - |\Phi_x| & \geq  |\Phi_{S_{n}}| - |\Phi_{(S_k)^{b} \times S_{n-bk}}| \\[1ex]
  &  =     n(n-1)-[bk(k-1)+(n-kb)(n-kb-1)]\\[1ex]
  & = 2kbn-k^2b^2 - bk^2.
\end{aligned} 
\]

We claim that $N(x)>b(n-1)$. Indeed, if $k=1$, then since $b<n$,
\[
N(x) = 2bn-b^2 -b > b(n-1). 
\]
On the other hand, if $k>1$, then
\[
N(x)= 2kbn-k^2b^2 - bk^2 \geq b(2kn - kn - k^2) > b(n-1). 
\]
Here, the first inequality follows from the fact that $bk\leq n$. The
second inequality is equivalent to
\[
n(k-1)>(k-1)(k+1)\iff n>k+1,
\]
 which is true because $k\leq \frac{n}{2}$ and $n\ge 3$. This completes the proof for $G=\SL_n$. \qed

 \subsection{Proof in the exceptional cases} 

Next, we turn our attention to the proof in the exceptional cases. We
provide the complete proof for $E_6$. The proof for the other
exceptional types is similar, but easier.

Recall that $E_6$ has $72$ roots and its degrees are $2$, $5$, $6$,
$8$, $9$, and
$12$.  It is easy to check that the proper parabolic of $E_6$ with the
largest number of roots is $D_5$ with $40$ roots. Thus, for all
non-zero $x\in V_\bC$, we have
\[
N(x)\geq 72-40=32>5\times6,
\]
so the claim holds for $b\leq 5$.

Now, assume $b>5$.  Let $Q$ denote the unique, up to scalar,
homogeneous quadratic invariant polynomial. If $x$ is an eigenvector
with eigenvalue a primitive $b^\mathrm{th}$ root of unity, then
$Q(x)=0$, because $b\nmid 2=\deg(Q)$. Using this fact, it is easy to show that the stabilizer $W_x$ must be a parabolic
subgroup of rank $\leq 4$, cf. \cite[Corollary 5]{Masoud}. 
The maximum number of roots in a parabolic subgroup of $E_6$ of rank
$\leq 4$ is $24$ (for $D_4$). Thus,
\begin{equation} \label{eq:E_6}
N(x)\geq 72-24=48>6\times 6,
\end{equation} 
so the result is also true for $b=6$. 

The remaining cases are $b=8$ and $b=9$. Let $x\in V(9)$ be a non-zero
element. One can show that if $x$ is not regular, then there exists a
proper parabolic subgroup of $W$ with a degree divisible by $9$. But
there is no such parabolic subgroup of $E_6$. Thus, $x$ must be
regular, and so the theorem is immediate.

It remains to treat the case $b=8$. Suppose $x$ is a non-regular
non-zero element of $V(8)$. One can show (cf. \cite[Lemma
4.12]{Springer}) that there exists a proper parabolic subgroup $P<
W$ and $w\in P$ such that 
\[
w\cdot x=\zeta x,
\]
where $\zeta$ is a primitive $8^\mathrm{th}$ root of unity.  

Now, the parabolic $P$ must have a degree divisible by $8$. The only
possibility is $P\simeq D_5$.  In this case, however, $8$ is the
\emph{highest} degree of $P$, and so, by a theorem of Kostant \cite[\S
9]{Kostant}, $x$ is \emph{regular} for the reflection action of $P$.
In particular,
 \[
 \a(x)\neq 0,\quad \quad  \textrm{for all roots $\alpha$ of $P$.}
 \]
It follows that  $\a(x)$ is zero for at most one simple root of $W$. Hence, either $x$ is regular or $W_x\simeq A_1$. In both cases, we have $N(x)>6\times 8$. \qed

 \subsection{Conjugacy classes over Laurent series}

 We record an application of Theorem \ref{t:reflection} to
 the study of rationality of conjugacy classes over Laurent series.

\begin{cor}\label{c:conjugacy}
Suppose that $Y\in\fg(\bcK)$ is conjugate to an
  element of $\fg(\cK)$ and that the semisimple part $X$ of $Y$ lies
  in $\fh(\bcK)$. Suppose 
  \[
  X=xt^{a/b} + \textrm{higher order terms}, \quad \quad  x\in \fh\setminus \{0\},\quad \gcd(a,b)=1.
  \]
    Then,
  $N(x)\geq br$. Moreover, equality is achieved if and only if $b=h$ is the Coxeter number, $Y=X$,
  and $x$ is a regular eigenvector of a Coxeter element of $W$.
\end{cor}

\begin{proof}  We first observe that it suffices to assume that $Y$ is
  semisimple.  Indeed, if $\Ad(g)(Y)\in\fg(\cK)$, then the semisimple
  part of $\Ad(g)(Y)$ is $\Ad(g)(X)$.  Since $\cK$ is perfect,
  $\Ad(g)(X)\in\fg(\cK)$~\cite[Theorem 4.4(2)]{Borel}. 

  In view of Theorem \ref{t:reflection}, to obtain the inequality, it
  is sufficient to show that $x\in V(b)$, i.e., $x$ is an eigenvector
  for some element of $W$ with eigenvalue a primitive $b^\mathrm{th}$
  root of unity.

  Let $Q\in \bC[\fg]^G$ be an invariant homogeneous polynomial of
  degree $d$. Note that $Q$ is also an invariant polynomial on
  $\fg(\cK)$ and $\fg(\bcK)$.  Choosing $g\in G(\bcK)$ such that
  $\Ad(g)(X)\in\fg(\cK)$, we see that $Q(X)=Q(\Ad(g)(X))\in\cK$.
 
  Since $x t^{\frac{a}{b}}$ is the leading term of $X$, we have
  have \[ Q(X)=Q(x t^{\frac{a}{b}}) + \textrm{higher order terms} =
  t^{\frac{da}{b}} Q(x) + \textrm{higher order terms}.
  \]
    Then, for the leading term of the
  above expression to be in $\cK$, we must have
  \begin{equation*} \label{eq:chevalley} 
  \textrm{$Q(x)=0\quad $ whenever $b$ does not divide $d$}.
  \end{equation*} 
Here, we are using the fact that $(a,b)=1$.  Since $\bC[\fg]^G \simeq
\bC[\fh]^W$, it now follows from \eqref{vbq} that $x\in V(b)$; thus, the inequality is established.   

It is immediate from Theorem~\ref{t:reflection} that $N(x)=br$ if and
only if $b=h$ and $x$ is a regular eigenvector of a Coxeter element.
This means that the leading term of $X$ is regular semisimple, which
implies that $X$ itself is regular semisimple.  In particular, $Y=X$.

\end{proof}

\begin{rem} \label{r:rationality} Take $X\in \fg(\bar{\cK})$, and let
  $\cO_X\subset \fg(\bar{\cK})$ denote the $G(\bar{\cK})$-orbit of $X$.
  Then $X$ is $G(\bar{\cK})$-conjugate to an element of $\fg(\cK)$ if
  and only if $\cO_X$ is closed under the action of
  $\Gal(\bar{\cK}/\cK)$, i.e., $\cO_X$ is defined over $\cK$.  To see
  this, note that the
  forward implication is trivial.  The converse follows from a theorem
  of Steinberg, stating that any homogeneous space defined over a
  perfect field of cohomological dimension $\le 1$ has a rational
  point~\cite[Theorem 1.9]{Steinberg}.
\end{rem}

\subsection{Gauge classes over Laurent series}
We record a version of the above corollary for gauge equivalence classes over Laurent series. 

\begin{prop}\label{p:conjugacy2}
  Let $\n$ be an irregular singular connection on
  $\cDt$ with Jordan form $d+(X+N)\duu$.  Let $xt^{a/b}$ be the
  leading term of $X$, with $x\in \fh\setminus \{0\}$ and
  $\gcd(a,b)=1$.  Then, $N(x)\geq br$. Moreover, equality is achieved
  if and only if $b=h$, $N=0$, and $x$ is a regular eigenvector of a
  Coxeter element of $W$.
\end{prop}

\begin{proof}  We will show that $X$ is conjugate  to an element of
  $\fg(\cK)$; the desired inequality will then follow by Corollary
  \ref{c:conjugacy}.
  Let us write
 \[
 X =\sum_{i} x_i t^{r_i}, \quad \quad x_i\in
 \fh\setminus \{0\}, \quad r_i\in \bQ_{\le 0}.
\] Since $d+(X+N)\duu$ is $G(\bar{\cK})$-gauge equivalent to the
pullback of a connection on $\cDt$, the proposition in \S 9.8 of
\cite{BV} implies that there exists an integer $c\geq 1$ and an
element $\theta\in G$ such that $\theta^c=1$, $cr_i=s_i\in \bZ$, and
 \[
 \Ad(\theta)(X)= \sum_i x_i t^{r_i} \omega_c^{s_i}=\sum_i x_i t^{s_i/c} \omega_c^{s_i}=\sigma_c(X).
 \]
 Here, $\omega_c=e^{2\pi i/c}$ and $\sigma_c$ is the generator of $\Gal(\cK_c/\cK)$ defined by $t^{1/c}\mapsto \omega_ct^{1/c}$. Applying $\theta$ repeatedly, we see that 
 \[
 \Ad({\theta^j})(X)=\sigma_c^j(X), \quad \quad j=0,\dots, c-1.
 \]

 Since the action of $\Gal(\bar{\cK}/\cK)$ on $X$ factors through
 $\Gal(\cK_c/\cK)$, we have shown that the Galois orbit of $X$ is
 contained in $\Ad(G(\bar{\cK}))(X)$.  It is now immediate that this
 adjoint orbit is defined over $\cK$, so by Remark
 \ref{r:rationality}, $X$ is conjugate to an element of $\fg(\cK)$.

 As in the proof of Corollary~\ref{c:conjugacy}, $N(x)=br$ holds if
 and only if $b=h$ and $x$ is a regular eigenvector of a Coxeter
 element, and this implies that $X$ is regular semisimple.  Since $N$
 commutes with $X$, we obtain $N=0$.
 \end{proof} 

 \begin{rem} The corollary in \S 9.8 of \cite{BV} gives a gauge
   version of Remark \ref{r:rationality}.  Let $\n=d+(X+N)\duu$ be a
   $G$-connection in Jordan form on $\cDb$ for some $b$; here,
   $u^b=t$.  Then $\n$ is $G(\bar{\cK})$-gauge equivalent to the
   pullback of a connection on $\cDt$ if and only if the gauge class
   of $\n$ is closed under the action of $\Gal(\bar{\cK}/\cK)$.
\end{rem}

\section{Proofs of the main theorems} \label{s:Proof}

\subsection{Proof of Theorem \ref{t:main}}

Let $(\cE,\nabla)$ be an irregular singular formal flat $G$-bundle. Let $d+(X+N)\duu$
denote the Jordan form of $\nabla$, where $u=t^{1/b}$,  $X\in
\fh(\bC[t^{-\frac{1}{b}}])$, and $N\in \fn(\bC)$. Let $xt^{-k/b}$ be the leading term
of $X$. Since $\n$ is irregular singular, $k>0$, so by Proposition
\ref{p:conjugacy2}, $N(x)\geq br$.

Now, we consider the adjoint connection. By uniqueness of the Jordan
form, $ d+(\Ad(X) + \Ad(N))\duu$ is the Jordan form for $\Ad(\nabla)$.
We view $\Ad(X)$ as a matrix in terms of a basis for $\fg$ consisting of weight
vectors.  In this basis, $\Ad(X)$ is clearly diagonal, and the
irregularity is the sum of the order of the poles of the diagonal
entries.  If $\a(x)\ne 0$, then the corresponding diagonal element has
a pole of order $k/b$, thereby contributing $k/b\ge 1/b$ to the
irregularity.   We thus obtain
\begin{equation}\label{eq:IrrN}
\Irr(\Ad(\nabla))\geq N(x)\frac{k}{b} \geq \frac{N(x)}{b} \geq r.
\end{equation}\qed

\subsection{Proof of Theorem \ref{t:main2}} The equivalence $(2)\iff
(3)$ and the implication $(3)\implies (1)$ have been established in
Proposition \ref{p:FG}. We will complete the proof of the theorem by
showing that $(1)\implies (2)$.  Suppose $\nabla$
is a connection with $\Irr(\Ad(\nabla))=r$. This means that  all the inequalities
in \eqref{eq:IrrN} are equalities. In particular, we have
$N(x)=br$, implying that $b=h$ and $k=1$.
Thus, $\nabla$ has slope $1/h$.  \qed

  \end{document}